\documentclass[11pt,a4paper,reqno]{amsart}
\usepackage{float, verbatim}
\usepackage{amsmath}
\usepackage{amssymb}
\usepackage{amsthm}
\usepackage{amsfonts}
\usepackage{mathtools}
\usepackage{enumitem}
\usepackage[colorlinks = true, citecolor = black, pagebackref=true]{hyperref}
\usepackage[abbrev,alphabetic]{amsrefs}
\usepackage{color}
\usepackage{xcolor}
\usepackage{enumerate}

\newtheorem*{thm*}{Theorem}

\newtheorem{thm}{Theorem}[section]

\newtheorem{lma}[thm]{Lemma}
\newtheorem{cor}[thm]{Corollary}
\newtheorem{defn}[thm]{Definition}

\newtheorem{prop}[thm]{Proposition}

\newtheorem{rem}[thm]{Remark}

\newcommand{\R}{\mathbb{R}}
\newcommand{\Z}{\mathbb{Z}}
\newcommand{\Q}{\mathbb{Q}}
\newcommand{\E}{\mathbb{E}}
\renewcommand{\P}{\mathbb{P}}
\newcommand{\T}{\mathbb{T}}
\newcommand{\C}{\mathbb{C}}

\newcommand{\s}{\sigma}
\newcommand{\e}{\varepsilon}
\renewcommand{\l}{\lambda}
\renewcommand{\a}{\alpha}
\renewcommand{\b}{\beta}
\newcommand{\g}{\gamma}

\newcommand{\cA}{\mathcal{A}}

\newcommand{\wh}{\widehat}
\newcommand{\wt}{\widetilde}

\DeclareMathOperator{\Id}{Id}
\DeclareMathOperator{\DC}{DC}
\DeclareMathOperator{\Gal}{Gal}

\usepackage{todonotes}

\begin{document}

\title[Fourier transform of self-similar measures]{ Fourier decay of self-similar measures and self-similar sets of uniqueness}
\author{P\'eter P. Varj\'{u}}
\address{P. P. Varj\'u, Centre for Mathematical Sciences, Wilberforce Road, Cambridge CB3~0WB, UK }
\curraddr{}
\email{pv270@dpmms.cam.ac.uk}
\thanks{}
\author{Han Yu}
\address{H. Yu, Centre for Mathematical Sciences, Wilberforce Road, Cambridge CB3~0WB, UK }
\curraddr{}
\email{hy351@maths.cam.ac.uk}
\thanks{ }

\subjclass[2010]{28A80, 42A16, 11A63}

\keywords{Rajchman measure; Fourier decay of self-similar measures; Sets of uniqueness; Digit changes}

\thanks{Both authors have received funding from the European Research Council (ERC) under the European Union’s Horizon 2020 research and innovation programme (grant agreement No. 803711). PV was supported by the Royal Society.}

\date{}

\dedicatory{}

\begin{abstract}
In this paper, we investigate the Fourier transform of self-similar measures on $\mathbb{R}$. We provide quantitative decay rates of Fourier transform of some self-similar measures. Our method is based on random walks on lattices and Diophantine approximation in number fields. We also completely identify all self-similar sets which are sets of uniqueness. This generalizes a classical result of Salem and Zygmund. 
\end{abstract}

\maketitle
\allowdisplaybreaks
\section{Introduction}

\subsection{Background}\label{Frac}
Let $\mathcal{F}=\{f_1,\dots,f_k\}$ be a set of $k\geq 2$ similarities on $\R$:
\begin{equation}\label{eq:IFS}
f_i(x)=r_ix+a_i
\end{equation}
with $0<r_i<1$ for each $i\in\{1,\dots,k\}$. Then by \cite{H81}, there is a unique set $F$ such that
\[
F=\bigcup_{i=1}^{k} f_i(F).
\]
This $F$ is called \emph{self-similar} and the attractor of the IFS $\mathcal{F}$.

In this paper, we will consider self-similar measures. Given a list of positive numbers $(p_1,\dots,p_k)$ with $\sum_{i=1}^k p_i=1$
(a \emph{positive probability vector}), there is a unique probability measure $\mu$ with
\[
\mu=\sum_{i=1}^k p_i f_i\mu.
\]
This measure $\mu$ is supported on $F$ and it is called a \emph{self-similar measure}.
Here, and everywhere in the paper, we write $f\mu$ for the pushforward of $\mu$ under $f$ that is
the measure that satisfies
$(f\mu)(A)=\mu(f^{-1}A)$ for all Borel sets $A$.

Let $\omega\in\mathbb{R}$ and we define the Fourier transform as
\[
\wh{\mu}(\omega)=\int_{\mathbb{R}} e(\omega x)d\mu(x).
\]
For convenience of notation, we write
\[
e(x)=e^{2\pi \mathbf{i} x}.
\]
If $|\wh{\mu}(\omega)|\to 0$ for $\omega\to 0$, we say that $\mu$ is a \emph{Rajchman measure}. 

Our understanding of the Rajchman property of self-similar measures is now almost complete thanks to results in
\cite{LS19} and \cite{B19}.

\begin{thm}[Li-Sahlsten \cite{LS19}]\label{th:LS}
	Let $\mu$ be a non-singleton self-similar measure on $\mathbb{R}$ associated to the IFS \eqref{eq:IFS} and a positive probability vector. If $\log r_i /\log r_j$ is irrational for some $i$ and $j$, then $\mu$ is Rajchman. Moreover, if $\log r_i /\log r_j$ is not Liouville for some $i$ and $j$, that is, there is some $C>0$ such that $|\log r_i /\log r_j-p/q|>q^{-C}$ for all $p,q\in \Z_{>0}$, then
	\[
	\wh\mu(\omega)=O(|\log|\omega||^{-c})
	\]
	for some $c>0$.
\end{thm}

This result reduces the problem of characterizing the Rajchman property of self-similar measures to the case when the contraction ratios
satisfy $r_j=r^{l_j}$ for some $r\in(0,1)$ and $l_j\in\Z_{>0}$. This has been analyzed by Br\'emont.

\begin{thm}[Br\'emont \cite{B19}]\label{th:B}
Let $k\ge 2$, $r\in(0,1)$, $l_1,\ldots,l_k\in\Z_{>0}$ with $\gcd(l_1,\ldots, l_k)=1$ and $r_i=r^{l_i}$ for $i=1,\ldots,k$.
Let $\mu$ be a non-singleton self-similar measure associated to the IFS \eqref{eq:IFS} and a positive
probability vector.
If $\mu$ is not Rajchman, then $r^{-1}$ is a Pisot number and the IFS can be conjugated by a suitable similarity to a form
such that $a_j\in\Q(r)$ for all $j$.
\end{thm}

We recall that a real number $\l>1$ is Pisot if it is an algebraic integer and all of its Galois conjugates have modulus less than $1$.

\begin{rem}
Br\'emont also shows that when the IFS satisfies the conclusion of the theorem, then there is a choice of the probability vector
that makes the self-similar measure non-Rajchman, and in fact this holds if the probability vector is outside a finite union of proper
submanifolds of the simplex of all probability vectors.
Moreover, he carried out a detailed analysis of the set of exceptional probability vectors that make the measure Rajchman.
We refer to his paper for more details.
\end{rem}

%%%%%%%%%%%%%%%%%%%%%%%%%%%%%%%%%%%%%%%%%%%%%%%%%%%%%%%%%%%%%
\subsection{Results in this paper}\label{sc:results}
%%%%%%%%%%%%%%%%%%%%%%%%%%%%%%%%%%%%%%%%%%%%%%%%%%%%%%%%%%%%%

In this paper, we consider the following two questions:
\begin{itemize}
	\item[Q1:] Which self-similar sets are sets of uniqueness? (We recall the definition below.)
	\item[Q2:] If a self-similar measure is Rajchman, what can we say about the decay rate of its Fourier transform at infinity?
\end{itemize}

A set $E\subset[0,1]$ is a \emph{set of uniqueness} if any trigonometric series
\[
	\sum_{n\geq 0} c_n e(n x),
\]
that converges to $0$ on $[0,1]\backslash E$ is trivial, that is $c_n=0$ for all $n$.
A set $E\subset\R$ is a set of uniqueness if its image in $\R/\Z\equiv[0,1)$ is.
Research on sets of uniqueness can be traced back to the work of Riemann, Cantor and others in the 19'th century
on the question, whether the limit function of a trigonometric series determines the coefficients uniquely.
This topic has a vast and rich literature.
We direct the reader to the book \cite{KL87} for a general reference on the subject.

To answer Question Q1, we extend the classical Salem-Zygmund theorem \cite{KL87}*{Theorem III.4.1} by identifying all self-similar sets which are sets of uniqueness.

\begin{thm}\label{Uniqueness}
A self-similar set is a set of uniqueness if and only if it has zero Lebesgue measure and its underlying IFS can be conjugated by a
suitable similarity to the form
\[
\{f_j(x)=r^{l_j}x+a_j:j=1,\ldots, k\},
\]
where $r^{-1}$ is a Pisot number, $l_1,\ldots,l_k\in\Z_{>0}$ with $\gcd(l_1,\ldots,l_k)=1$ and $a_j\in\Q(r)$.
\end{thm}

The necessity of the conditions in this theorem is not new.
Indeed, it is well known (see \cite{KL87}*{Proposition I.3.1}) that sets of uniqueness always have zero Lebesgue measure.
Moreover, Piatetski-Shapiro and in another work Kahane and Salem (see \cite{KL87}*{Theorem II.4.1})
gave a criterion for a closed
set being a set of uniqueness in terms of the distributions supported on the set.
They proved that a closed set $E$ is a set of uniqueness if and only if it does not support a non-zero pseudofunction.
For the definition of pseudofunctions and for more background we refer to \cite{KL87}*{Chapter II}.
For our purposes it is enough to note that a probability measure is a pseudofunction if and only if it is Rajchman.
This means that a set of uniqueness cannot support a Rajchman measure, hence the necessity of the condition on the IFS in Theorem
\ref{Uniqueness}
follows from Theorems \ref{th:LS} and \ref{th:B}.

For question Q2, we prove the following result.
\begin{thm}\label{MAIN}
	Let $k\geq 2$ be an integer. Let $r_1=r^{l_1},\dots,r_k=r^{l_k}$ for some $r\in(0,1)$ and $l_1,\ldots,l_k\in\Z_{>0}$
	with $\gcd(l_1,\ldots,l_k)=1$. Furthermore, let $a_1,\dots,a_k\in\mathbb{R}$. We consider the IFS \eqref{eq:IFS}. Assume that $l_1=l_2$. Let $(p_1,\ldots,p_k)$ be a probability vector with strictly positive entries and write $\mu$ for the corresponding self-similar measure.
	
	Then there is a constant $C>0$ depending only on $l_1,\ldots,l_k$ and $p_1,\ldots,p_k$ such that
	\[
	|\wh\mu(\omega)|\le\exp\Big(-C^{-1}\sum_{j>C}\|(a_1-a_2)\omega r^j\|^2)\Big),
	\]
	where $\|\cdot\|$ denotes distance to the nearest integer.
\end{thm}

The assumption $l_1=l_2$ may seem restrictive.
However, it is always possible to realize $\mu$ as a self-similar measure associated to an IFS that contains
similarities with equal contraction ratios.
See Lemma \ref{lm:newIFS} for a construction.

Sums appearing in the right hand side of the conclusion have been studied extensively in the literature.
We observe, that Theorem \ref{MAIN} implies Theorem \ref{th:B} thanks to a classical theorem of Pisot.
We will give the details in Section \ref{sc:appl}.

The following can be deduced from a result of Bufetov and Solomyak \cite{BS14}*{Proposition 5.5} and Theorem \ref{MAIN}.

\begin{cor}\label{cr:BS}
Let $k\geq 2$ be an integer. Let $r_1=r^{l_1},\dots,r_k=r^{l_k}$ for some $r\in(0,1)$ and $l_1,\ldots,l_k\in\Z_{>0}$ with $\gcd(l_1,\ldots,l_k)=1$. Assume that $r^{-1}$ is an algebraic integer that is not a Pisot or Salem number. Let $\mu$ be a non-singleton self-similar measure associated to the IFS \eqref{eq:IFS} and a positive probability measure. Then
\[
|\wh\mu(\omega)|=O(|\log|\omega||^{-c})
\]
for some $c>0$.
\end{cor}

Recall that a real number $\l>1$ is Salem, if it is an algebraic integer, all of its Galois conjugates have modulus at most $1$,
and at least one of them has modulus $1$.

A special case of this result for homogeneous self-similar measures has been proved by Gao and Ma
\cite{GM-Fourier-decay}*{Theorem 1.1}.

In the next section, we state some results in number theory that yield another application of Theorem \ref{MAIN}.
Before we can state this, we need to introduce some terminology.

\begin{defn}\label{df:Liou}
Let $K$ be a number field. We say that $\alpha\in\mathbb{R}\setminus K$ is Liouville over $K$, if for any integer $H>0$ there exists $\beta\in K$ such that
\[
|\alpha-\beta|\leq \frac{1}{e^{H h(\beta)}}.
\]
Here, and everywhere in this paper, $h(\beta)$ stands for the absolute logarithmic height of $\beta$.
That is, if $P(x)=a_d\prod(x-\b_j)\in\Z[x]$ is the minimal polynomial of $\beta$ with leading coefficient $a_d$ and complex
roots $\b_j$, then
\[
h(\b)=\log|a_d|+\sum\max(0,\log|\b_j|).
\]
For more on heights, we refer to \cite{Mas}*{Chapter 14}.
\end{defn}

We note that algebraic numbers outside $K$ are not Liouville, as follows from the Liouville inequality, see e.g. \cite{Mas}*{Proposition 14.13}.

\begin{cor}\label{cr:Liou}
Let $k\geq 3$ be an integer. Let $r_1=r^{l_1},\dots,r_k=r^{l_k}$ for some $r\in(0,1)$ and $l_1,\ldots,l_k\in\Z_{>0}$ with $\gcd(l_1,\ldots,l_k)=1$. Furthermore, let $a_1,\dots,a_k\in\mathbb{R}$. Suppose $a_1=0$ and
$a_2/a_3$ is not Liouville over $\Q(r)$. Then
\[
|\wh\mu(\omega)|=O(|\log|\omega||^{-c})
\]
for some $c>0$.
\end{cor}

We note that the condition $a_1=0$ is just a choice of normalization, it does not restrict generality.
In light of Corollary \ref{cr:BS}, this result is interesting only when $r^{-1}$ is a Pisot or Salem number.

%%%%%%%%%%%%%%%%%%%%%%%%%%%%%%%%%%%%%%%%%%%%%%%%%%%%%%%%%%%%
\subsection{A result in number theory}
%%%%%%%%%%%%%%%%%%%%%%%%%%%%%%%%%%%%%%%%%%%%%%%%%%%%%%%%%%%%

In this section, we discuss a result in Diophantine approximation in number fields, which we will use to derive Corollary \ref{cr:Liou}. We refer to the book \cite{B} for related literature.

\begin{defn}
Let $\lambda>1$ be a real number. Let $x>0$ be a real number. For any real number $\e>0$, we define
\[
\DC(x,\lambda,\e)=\#\{t\in\Z_{\ge 0}:x\lambda^{-t}\geq 1, \|x\lambda^{-t}\|>\e         \},
\] 
where $\|x\|$ is the distance form $x$ to its nearest integer.
\end{defn}
We refer to numbers $t$ that satisfy the condition in the definition of $\DC(x,\lambda,\e)$ as the places of digit changes of $x$ in base $\lambda$ with precision $\e$. Suppose that $\lambda=2$. Then $x\lambda^{-t}$ is close to an integer if the binary expansion of $x$ has a long consecutive block of $0$ or $1$ at the $t$-th position (counting from the left). We take the terminology 'digit change' from here.

\begin{thm}\label{AG2}
	Let $\lambda>1$ be an algebraic number. Let $\gamma$ be a number which is not Liouville over $\Q(\lambda)$. Then there are a number $\e>0$ and a constant $C$ such that for all $x>e^e$
	\[
	\DC(x,\lambda,\e)+\DC(x\gamma,\lambda,\e)\geq C \log\log x.
	\]
\end{thm}
\begin{rem}
	By a result due to Bufetov and Solomyak \cite{BS14}*{Proposition 5.5}, if $\lambda$ is neither Pisot nor Salem, we have
	\[
	\DC(x,\lambda,\e)\geq C \log\log x,
	\]
	where $\e,C$ are constants depending on $\lambda$. So this result is new only when $\l$ is Pisot or Salem.
\end{rem}

%%%%%%%%%%%%%%%%%%%%%%%%%%%%%%%%%%%%%%%%%%%%%%%%
\subsection{Organization of the paper}
%%%%%%%%%%%%%%%%%%%%%%%%%%%%%%%%%%%%%%%%%%%%%%%%
We prove Theorem \ref{MAIN} in Section \ref{sc:decay}. 
We prove Theorem \ref{AG2} in Section \ref{sc:DC}.
We discuss the applications of Theorem \ref{MAIN} stated above in Section \ref{sc:appl}. 
Finally, we prove Theorem \ref{Uniqueness} in Section \ref{sc:uniq}.

%%%%%%%%%%%%%%%%%%%%%%%%%%%%%%%%%%%%%%%%%%%%%%%%%%%%%%%
\section{Fourier decay of self-similar measures}\label{sc:decay}
%%%%%%%%%%%%%%%%%%%%%%%%%%%%%%%%%%%%%%%%%%%%%%%%%%%%%%%

The purpose of this section is to prove Theorem \ref{MAIN}.
Recall our IFS has scaling factors $r^{l_1},\ldots,r^{l_k}$ with $\gcd(l_1,\ldots,l_k)=1$.
We assume $l_1=l_2$.
The translations are $a_1,\ldots,a_k\in\R$ and the positive probability vector is $(p_1,\ldots,p_k)$.

We first note the formula
\begin{equation}\label{eq:ssF}
\wh\mu(\omega)=\sum_{i=1}^k p_i\wh\mu(r^{l_i}\omega)e(a_i\omega),
\end{equation}
which is just the self-similarity relation expressed in terms of the Fourier transform.
Using $l_1=l_2$, this implies
\begin{equation}\label{eq:ssF2}
|\wh\mu(\omega)|\le \sum_{i=2}^k \wt p_i|\wh\mu(r^{l_i}\omega)|w_i(r^{l_i}\omega),
\end{equation}
where
\[
\wt p_i =
\begin{cases}
p_1+p_2, & \text{if $i=2$}\\
p_i, & \text{if $i>2$}
\end{cases}
\]
and
\[
w_i(r^{l_i}\omega) = 
\begin{cases}
\frac{|p_1e(a_1\omega)+p_2e(a_2\omega)|}{p_1+p_2}, &\text{if $i=2$}\\
1, & \text{if $i\ge2$.}
\end{cases}
\]

We express \eqref{eq:ssF2} in probabilistic notation, which will allow us to generalize it in a convenient way.
We write $I_1,I_2,\ldots$ for a sequence of independent random indices drawn form $\{2,\ldots,k\}$ with
probabilities $\{\wt p_2,\ldots, \wt p_k\}$.
We fix some $\omega\in\R$ and introduce the random variables
\[
X_n= \Big(\prod_{j=1}^n r^{l_{I_j}}\Big)\omega
\]
and
\[
W_n=\prod_{j=1}^n w_{I_j}(X_j).
\]
With this notation, \eqref{eq:ssF2} can be written as
\[
|\wh\mu(\omega)|\le \E[W_1|\wh\mu(X_1)|],
\]
and more generally, applying \eqref{eq:ssF2} for $X_n$ in place of $\omega$ we can write
\[
|\wh\mu(X_{n})|\le\E\Big[\frac{W_{n+1}}{W_{n}}|\wh\mu(X_{n+1})|\;\Big|\;I_1,\ldots, I_n\Big],
\]
which we rewrite as
\[
W_n|\wh\mu(X_{n})|\le\E[W_{n+1}|\wh\mu(X_{n+1})|\;|\;I_1,\ldots, I_n].
\]
This means that $W_n|\wh\mu(X_n)|$ is a submartingale, and the following is an immediate consequence of the
Optional Stopping Theorem, see e.g. \cite{BW-basic-probability}*{Theorem 3.6(c) and Remark 3.6}.

\begin{lma}\label{lm:stopping}
Let $\tau$ be a stopping time, that is, a random variable such that the event $\tau=n$ is measurable with respect to the
$\s$-algebra generated by $I_1,\ldots, I_n$ for all $n$.
Assume $\P(\tau<\infty)=1$.
Then
\[
|\wh\mu(\omega)|\le\E[W_\tau|\wh\mu(X_\tau)|].
\]
\end{lma}

Before going further, we explain our strategy to prove Theorem \ref{MAIN} informally.
Since $|\wh\mu(X_\tau)|\le 1$ always, it is enough to estimate $\E[W_\tau]$
to get a bound on $|\wh\mu(\omega)|$.
Each time we have $I_j=2$, the value of $W$ decreases by a factor $w_2(X_j)$.
In what follows we will argue that the random walk $X_j$ will hit any point of the form $r^t\omega$
with some substantial probability if $t$ is not too small, and we will try to estimate the aggregated
effect of the $w_2$ factors.
Our main tool is the following result.

\begin{thm}[Erd\H{o}s--Feller--Pollard \cite{EFP49}, see also \cite{B53}]\label{EFP}
Consider the random walk such that $Y_0=0$ and $Y_{i+1}=Y_i+l_{I_i}$.
For each $t\geq 0$, define $P_t$ to be the probability that $Y_i$ visits $t$ for some $i\ge0$. Then we have
	\[
	\lim_{t\to\infty} P_t= \frac{1}{\E[l_{I_1}]}.
	\]
\end{thm}

We write $\tau(t)$ for the first time $n$ such that $X_n\le r^t\omega$.
This is clearly a stopping time.

\begin{lma}\label{lm:conditional}
There is a constant $C>0$ that depends only on the distribution of $l_{I_1}$ such that the following holds.
For all $t>C$, we have
\[
\E\Big[\frac{W_{\tau(t)}}{W_{\tau(t-C)}}\Big|I_1,\ldots,I_{\tau(t-C)}\Big]\le 1-C^{-1}(1-w_2(r^t\omega)). 
\]
\end{lma}

\begin{proof}
We write $\wt Y_n=\log_r(X_n/\omega)$.
We first observe that $\wt Y_{\tau(t-C)}\le t-C+\max(l_i)$.
(Note that $r<1$, so $\log_r$ is a decreasing function, hence  $\wt Y_{\tau(t)}\ge t$ for all $t$.)
By Theorem \ref{EFP}, (applied with $Y_m=\wt Y_{\tau(t-C)+m}-\wt Y_{\tau(t-C)}$),
there is a constant $C_0$ depending only on the distribution of $l_{I_1}$
such that
\[
\P(\wt Y_n=t-l_2\text{ for some $n$}|I_1,\ldots,I_{\tau(t-C)})\ge C_0^{-1}
\]
provided $C-2\max(l_i)\ge C_0$.
In this case, we also have
\[
\P(\wt Y_{\tau(t)}=t\text{ and }I_{\tau(t)}=2|I_1,\ldots,I_{\tau(t-C)})\ge C_0^{-1}\wt p_2,
\]
which yields
\[
\P\Big[\frac{W_{\tau(t)}}{W_{\tau(t-C)}}\le\frac{W_{\tau(t)}}{W_{\tau(t-1)}}=w_2(r^t\omega)
\Big|I_1,\ldots,I_{\tau(t-C)}\Big]\ge C_0^{-1}\wt p_2.
\]
This proves the lemma, provided we choose $C$ sufficiently large so that $C\ge C_0+2\max(l_i)$
and $C^{-1}\le C_0^{-1}\wt p_2$.
\end{proof}

\begin{prop}\label{pr:quant}
Let
\[
u(\omega)=1-\frac{|p_1e(a_1\omega)+p_2e(a_2\omega)|}{p_1+p_2}.
\]
Then we have
\[
|\wh\mu(\omega)|\le\prod_{n\ge C}(1-C^{-1}u(r^n\omega)).
\]
\end{prop}

We note that
\[
u(\omega)\ge C^{-1}\|(a_1-a_2)\omega\|^2
\]
for a suitable constant depending on $p_1$ and $p_2$, which can be seen by expanding the function $e(\cdot)$
in Taylor series.
Therefore, Theorem \ref{MAIN} follows immediately form the proposition.

\begin{proof}[Proof of Proposition \ref{pr:quant}]
We denote by $C_0$ the constant $C$ in Lemma \ref{lm:conditional}, which can be chosen to be an integer.
We note that
\[
1-w_2(\omega)=u(r^{-l_2}\omega).
\]

Now for any $t>C_0$, we have
\begin{align*}
\E[W_{\tau(t)}]=&\E\Big[\E\Big[\frac{W_{\tau(t)}}{W_{\tau(t-C_0)}}\Big|I_1,\ldots,I_{\tau(t-C)}\Big]W_{\tau(t-C_0)}\Big]\\
\le& (1-C_0^{-1}u(r^{t-l_2}\omega))\E[W_{\tau(t-C_0)}].
\end{align*}
Iterating this inequality and noting that $\E[W_{\tau(t-C_0)}]\le 1$ always, we get
\[
\E[W_{\tau(t)}]\le\prod_{j=0}^{\lfloor t/C_0\rfloor-1}(1-C_0^{-1}u(r^{t-jC_0-l_2}\omega)).
\]
Now we fix some $s\in\Z_{>0}$ and multiply this inequality together for $t=s,\ldots,s+C_0-1$ and get
\[
\prod_{t=s}^{s+C_0-1}\E[W_{\tau(t)}]\le\prod_{j=C_0+1}^{s+C_0-1} (1-C_0^{-1}u(r^{j-l_2}\omega)).
\]

By Lemma \ref{lm:stopping} and $|\wh\mu(\omega)|\le 1$, we have
\[
|\wh\mu(\omega)|\le\E[W_{\tau(t)}],
\]
so
\[
|\wh\mu(\omega)|^{C_0}\le \prod_{j=C_0+1}^{s+C_0-1} (1-C_0^{-1}u(r^{j-l_2}\omega))
\]
for all $s$,
and the result follows with a suitable choice of $C$.
\end{proof}

%%%%%%%%%%%%%%%%%%%%%%%%%%%%%%%%%%%%%%%%%%%%%%%%%%%%%%%%%%%
\section{Digit changes}\label{sc:DC}
%%%%%%%%%%%%%%%%%%%%%%%%%%%%%%%%%%%%%%%%%%%%%%%%%%%%%%%%%%%

The purpose of this section is to prove Theorem \ref{AG2}.
We will prove the following lemma.
\begin{lma}\label{LIO}
	Let $\lambda>1$ be an algebraic number. Let $\alpha\in [1,\lambda]$ be a real number and $\alpha\notin\mathbb{Q}(\lambda)$. Then there are numbers $\e, H, C_1,C_2>0$ depending only on $\lambda$ such that if $n,K>H$ and
	\begin{equation}\label{eq:NDC}
	\max_{n\leq j\leq Kn} \{\|\alpha\lambda^j\|\}\leq \e
	\end{equation}
	then there is a number $\beta\in\mathbb{Q}(\lambda)$ with height at most $h(\b)\le C_1 n$ such that
	\[
	|\alpha-\beta|\leq \frac{1}{e^{C_2 K n}}.
	\]
\end{lma}

We note that the lemma follows trivially from the aforementioned result of Bufetov and Solomyak \cite{BS14}*{Proposition 5.5}
when $\l$ is not a Pisot or Salem number.
Our proof is valid for all algebraic numbers.

For proving the above lemma, we need the following standard result. For completeness, we provide a detailed proof.

\begin{lma}\label{First Column}
	Let $\lambda$ be an algebraic number over $\mathbb{Q}$. Suppose that $\deg \lambda=d\geq 1.$ Let $\l_1=\l,\l_2\dots,\l_d$ be the Galois conjugates of $\l$. Consider the $d\times d$ Vandermonde matrix $A$ with entries
	\[
	A_{i,j}=\l^{j-1}_{i}, 1\leq i,j\leq d.
	\]
	Then $A$ is nonsingular, and its inverse $A^{-1}$ satisfies for each $i,j$ that
	\[
	(A^{-1})_{i,j}\in\mathbb{Q}(\l_j).
	\]
	In particular, the first column of $A^{-1}$ ($(A^{-1})_{i,j}, j=1$) has entries in $\mathbb{Q}(\lambda)$.
\end{lma}

\begin{proof}
	The fact that $A$ is nonsingular follows from the fact that $\l_1,\dots,\l_d$ are distinct.
	
We write $G=\Gal(E/\Q)$, where $E$ is the Galois closure of $\Q(\l)$. 
We write $G_i$ for the subgroup of $G$ consisting of those elements that fix $\lambda_i$.
The fixed field of $G_i$ certainly contains $\Q(\l_i)$.
On the other hand $[G:G_i]=d=[\Q(\l_i):\Q]$ by the orbit-stabilizer theorem, so the fixed field of $G_i$
equals $\Q(\l_i)$.
This means that an element $x\in E$ is in $\Q(\l_j)$ if and only if $\s(x)=x$ for all $\s\in G_j$.

We fix $\s\in G_j$ and set out to prove
\[
\s((A^{-1})_{i,j})=(A^{-1})_{i,j}.
\]
	By Cramer's rule, we know that
	\[
	(A^{-1})_{i,j}=\det(A)^{-1} C_{i,j},
	\]
	where $C_{i,j}=(-1)^{i+j}\det(A^*(j,i))$ and $A^*(j,i)$ is the $(j,i)$-minor of $A$, i.e. it is the matrix formed by the entries of $A$ after deleting the $i$-th column and the $j$-th row.  
	
	As $\det(A)$ is a polynomial over the entries of $A$, we see that
	\[
	\sigma(\det(A))=\det(\sigma(A)),
	\]
	where $\sigma(A)$ is the matrix $\sigma(A)_{i,j}=\sigma(A_{i,j})$.
	Note that $\sigma$ acts as a permutation on $\{\l_1,\dots,\l_d\}$ fixing $\l_j$.
	We see that $\sigma(A)$ is obtained by permuting the rows of $A$. Thus we have
	\[
	\sigma(\det(A))=\mathrm{sign} (\sigma) \det(A),
	\]
	where $\mathrm{sign}(\sigma)\in\{\pm 1\}$ is the parity of $\sigma$ as a permutation on $\{\l_1,\dots,\l_d\}$.
	
	Since $\sigma$ fixes $\l_j$, we see that the parity of $\sigma$ as a permutation on the subset $\{\l_1,\dots,\l_d\}\setminus \{\l_j\}$ is equal to $\mathrm{sign}(\sigma)$. From this observation, we see that
	\[
	\sigma(C_{i,j})=\mathrm{sign}(\sigma) C_{i.j}.
	\]
	Thus, we have
	\[
	\sigma((A^{-1})_{i,j})=\sigma(\det(A))^{-1}\sigma(C_{i,j})=\det(A)^{-1} C_{i,j}=(A^{-1})_{i,j},
	\]
as required.
This completes the proof of the lemma.
\end{proof}

\begin{proof}[Proof of Lemma \ref{LIO}]
	For each $j\in \{n,\dots,Kn\}$, we write
	\begin{equation}\label{eq:Keps}
	\alpha \lambda^j=K_j+\e_j,
	\end{equation}
	where $K_j$ is an integer and $|\e_j|\leq 1/2$.
	By our assumption \eqref{eq:NDC}, we have $|\e_j|\leq\e$ for all $j$.
	
	Since $\lambda$ is an algebraic number, we have
	\[
	\sum_{j=0}^d c_j\lambda^j=0,
	\]
	where $c_0,\dots,c_d$ are integers.
	Let $n',n'+1,\dots,n'+d$ be $d+1$ consecutive integers in $\{n,\dots,Kn\}$. 
	An appropriate linear combination of the equations \eqref{eq:Keps} gives
	\[
	0=\alpha \lambda^{n'} \sum_{j=0}^d c_j\lambda^j=\sum_{j=0}^d c_j K_{n'+j}+\sum_{j=0}^d c_j \e_{n'+j}.
	\]
	Assuming
	\[
	\e_{n'+j}\leq\e<(|c_0|+\ldots+|c_d|)^{-1},
	\]
	as we may, we then have
	\[
	\Big|\sum_{j=0}^d c_j \e_{n'+j}\Big|<1.
	\]
	This forces
	\[
	\sum_{j=0}^d c_j K_{n'+j}=0
	\]
	since the LHS is an integer.
	This in turn implies that
	\[
	\sum_{j=0}^d c_j \e_{n'+j}=0.
	\]

We can apply the above argument for every $d+1$ consecutive integers.
We see that $\{\e_{n+i}\}_{i=0,\dots,(K-1)n}$ forms a linear recurrence sequence
with coefficients $c_0,\dots,c_d$, which are integers.
As a well known fact, we see that there are complex numbers $b_1,\dots,b_d$
and $z_1,\dots,z_d$ such that
\[
\e_{n+i}=\sum_{j=1}^d b_j z^i_j.
\]
The complex numbers $z_1=\lambda,\dots,z_d$ are precisely the Galois conjugates of $\lambda$.

Let $A$ be a $d\times d$ matrix with entries $A_{j,m}:=z_j^{m-1}$.
Then we can write
\[
\e_{n+i+(m-1)}= \sum_{j=1}^d b_j z^i_j\cdot A_{j,m},
\]
and hence
\[
(\e_{n+i},\dots,\e_{n+i+d-1})=(b_1z_1^i,\ldots,b_dz_d^i) A
\]
for $i\in\{0,\dots,(K-1)n-d+1\}$.
Note that $A$ is not singular by Lemma \ref{First Column}.
Therefore, we can write
\[
(b_1z_1^{(K-1)n-d+1},\ldots,b_dz_d^{(K-1)n-d+1})=(\e_{Kn-d+1},\dots,\e_{Kn})A^{-1}
\]
and conclude that
\begin{equation}\label{eq:b1}
|b_1|<C|\l|^{-(K-1)n}
\end{equation}
for a constant $C$ that depends only on $\l$.
(Recall that $z_1=\l$.)

We can also write
\begin{equation}\label{eq:b1bd}
(b_1,\ldots,b_d)=(\e_n,\ldots,\e_{n+d-1})A^{-1}.
\end{equation}
Recall that $\e_i=\a\l^{n+i}-K_i$.
Observe that
\[
(1,0,\ldots,0)A=(1,\ldots,\l^{d-1}),
\]
hence
\[
(\a\l^{n},\ldots,\a\l^{n+d-1})A^{-1}=(\a\l^n,0,\ldots,0).
\]
We plug in these two facts in \eqref{eq:b1bd} and get
\[
(b_1,\ldots,b_d)=(\a\l^n,0,\ldots,0)-(K_n,\ldots, K_{n+d-1})A^{-1}.
\]

Now we set $\b$ such that $\b\l^n$ is the first entry of the vector $$(K_n,\ldots, K_{n+d-1})A^{-1}.$$
By Lemma \ref{First Column}, we see that the entries of the first column of $A^{-1}$ are in $\Q(\l)$ and they depend only on $\l$.
The numbers $K_n,\ldots, K_{n+d-1}$
are integers of absolute value at most $\l^{n+d}+1$.
Here we used that $\a\in[1,\l]$.
This means that $\b\in\Q(\l)$ and $h(\b)\le C_1 n$ for some constant $C_1$ depending only on $\l$.
Finally, we observe that
\[
|\a-\b|=\frac{|b_1|}{\l^n},
\]
and the claim of the lemma follows from \eqref{eq:b1}.
\end{proof}

\begin{proof}[Proof of Theorem \ref{AG2}]
Let $\lambda,\gamma$ be as given in the statement. Remember that $\g\notin\mathbb{Q}(\lambda)$.
Let $\e, H,C_1,C_2$ be as in Lemma \ref{LIO}.

For each $x>e^e$, we define $x_*$ to be the number in $(1,\lambda]$ with
\[
\log(x/x_*)/\log \lambda\in\mathbb{Z}.
\]
Let $K>H$ be an integer, and define the sequence $n_1=H+1$, $n_{j+1}=Kn_j+1$ for $j\ge 2$.
Write $\cA(x)$ for the set of integers $j\ge1$ such that $x_*\l^{n_{j+1}-1}\le x$ and \eqref{eq:NDC} is not
satisfied for $\alpha=x_*$, $K$ and all $n=n_j$.
Then it is straightforward that
\[
\DC(x,\lambda,\e)\ge |\cA(x)|.
\]

Write $N$ for the largest $j$ so that both $x_*\l^{n_{j+1}-1}\le x$ and $(\g x)_*\l^{n_{j+1}-1}\le \g x$ holds.
Note that $N\ge c\log\log x$ for a suitable constant $c>0$.
Now assume for contradiction that the conclusion of the theorem fails, with a constant $C$ small enough so that
\[
|\cA(x)|+|\cA(\g x)|\le \DC(x,\lambda,\e)+\DC(x,\lambda,\e)< N
\]
and with an $\e$ small enough so that Lemma \ref{LIO} applies.
This means that there is some $j$ such that $n_j\notin \cA(x)\cup\cA(\g x)$ but both both $x_*\l^{n_{j+1}-1}\le x$ and $(\g x)_*\l^{n_{j+1}-1}\le \g x$ holds.
Therefore, we can apply Lemma \ref{LIO} for $x_*$ and $(\g x_*)$ in place of $\a$ and $n=n_j$.
We conclude that there are $\b_1$ and $\b_2$ of height at most $C_1n$ such that
\[
x_*=\beta_1+O(1/e^{C_2Kn}), (\gamma x)_*=\beta_2+O(1/e^{C_2Kn}).
\]
We note that $\g=\l^k(\g x)_*/x$ for a suitable $k$, which is bounded in terms of $\g$ and $\l$, and
\[
\gamma=\l^k\frac{\beta_2+O(1/e^{C_2Kn})}{\beta_1+O(1/e^{C_2Kn})}=\l^k\frac{\beta_2}{\beta_1}+O(1/e^{C_2Kn}).
\]
Since $K$ is arbitrary, this shows that $\gamma$ is Liouville over $\Q(\l)$, a contradiction.
\end{proof}

%%%%%%%%%%%%%%%%%%%%%%%%%%%%%%%%%%%%%%%%%%%%%%%%%%%%%%%%%%%
\section{Applications of Theorem \ref{MAIN}}\label{sc:appl}
%%%%%%%%%%%%%%%%%%%%%%%%%%%%%%%%%%%%%%%%%%%%%%%%%%%%%%%%%%%

In this section, we deduce the various corollaries of Theorem \ref{MAIN},
which we claimed in Section \ref{sc:results}.
We first give a lemma that constructs an IFS to which Theorem \ref{MAIN} can be applied.

\begin{lma}\label{lm:newIFS}
Let $r_1=r^{l_1},\ldots,r_k=r^{l_k}$ for some $r\in(0,1)$ and $l_1,\ldots,l_k\in\Z_{>0}$ with $\gcd(l_1,\ldots,l_k)=1$.
Let $a_1,\ldots,a_k\in\R$.
Assume $a_0=0$.
Let $\mu$ be the self-similar measure associated to the IFS
\[
\{f_j:x\mapsto r_j x+a_j|j=1,\ldots,k\}
\]
and a positive probability vector $p_1,\ldots, p_k$.
Then for all $j$, there is some $k'\in\Z_{\ge 2}$, $l_1',\ldots, l_{k'}'\in\Z_{>0}$ with $\gcd(l_1',\ldots,l'_{k'})=1$
and $a_1',\ldots,a_{k'}'\in\R$,
such that $l_1'=l_2'$, $a_1'-a_2'=ba_j$ for some $0\neq b\in\Q(r)$ and $\mu$ is the self-similar measure associated to the IFS
\[
\{f_j':x\to r^{l_j'}+a_j'|j=1,\ldots,k'\}
\]
and a positive probability measure.
\end{lma}

\begin{proof}
If $l_1=l_j$, then it is enough to relabel the IFS exchanging the indices $2$ and $j$.

If $l_1\neq l_j$, then let $p$ be a prime such that $p\nmid l_j-l_1$, and consider the
IFS
\begin{equation}\label{eq:newIFS}
\{f_{i_1}\circ\ldots \circ f_{i_p}:i_1,\ldots,i_p\in\{1,\ldots,k\}\}
\end{equation}
and the positive probability vector
\[
\{p_{i_1}\cdots p_{i_p}:i_1,\ldots,i_p\in\{1,\ldots,k\}\}.
\]
Clearly, $\mu$ is the self-similar measure associated to this data.

The numbers
\[
\{l_{i_1}+\ldots+l_{i_p}:i_1,\ldots,i_p\in\{1,\ldots,k\}\}
\]
include $pl_i$ for $i=1,\ldots,k$, so their $\gcd$ is either $p$ or $1$.
However, $p\nmid(p-1)l_1+l_j$, so the $\gcd$ must be 1.

We label the maps in the new IFS \eqref{eq:newIFS} in such a way that the first two maps are
$f_1^{\circ(p-1)}\circ f_j$ and $f_2\circ f_1^{\circ(p-1)}$.
These clearly have the same contraction factor and the difference between their translation components is
\[
r^{(p-1)l_1}a_j-a_j,
\]
which is of the required form.
\end{proof}

%%%%%%%%%%%%%%%%%%%%%%%%%%%%%%%%%%%%%%%%%%%%%%%%%%%%%%%%%%%
\subsection{Theorem \ref{MAIN} implies Theorem \ref{th:B}}
%%%%%%%%%%%%%%%%%%%%%%%%%%%%%%%%%%%%%%%%%%%%%%%%%%%%%%%%%%%

We consider an IFS with contraction ratios $r_i=r^{l_i}$ for some $r\in(0,1)$, $l_1,\ldots,l_k\in\Z_{>0}$ with $\gcd(l_1,\ldots,l_k)=1$
as in Theorem \ref{th:B}.
By translating the coordinate system, we can assume that $0$ is the fixed point of the first map in the IFS, that is $a_0=0$.
We assume that the self-similar measure $\mu$ is not Rajchman and we aim to show that $r^{-1}$ is Pisot and the
IFS can be conjugated to a from such that $a_j\in\Q(\l)$ for all $j$.

Our proof is based on the following classical theorem of Pisot.

\begin{thm}[Pisot, \cite{B}*{Theorem 2.1}]\label{Pisot}
	Let $\lambda>1$ and $b\neq 0$ be real numbers such that
	\[
	\sum_{j=0}^{\infty} \|ba^j\|^2<\infty.
	\]
	Then $\lambda$ is a Pisot number and $b\in\mathbb{Q}(\lambda)$.
\end{thm}

Let $\omega_1,\omega_2,\ldots\in\R_{\ge1}$ be a sequence such that $\lim \omega_n=\infty$ and $|\omega_n|\ge\e$
for some $\e>0$ for all $n$.
For each $n$, define $k_n\in\Z_{\ge 0}$ and $\omega_n^{*}$ so that
\[
\omega_n^*=r^{k_n}\omega_n\in(r,1].
\]
By passing to a subsequence if necessary, we may assume that $\omega_n^*$ converges.
Furthermore, by rescaling the coordinate system if necessary, we may assume that $\lim\omega_n^*=1$.

We choose some $j\in(2,\ldots,k)$ so that $a_j\neq a_1=0$.
Such a $j$ exists otherwise $\mu$ is a singleton.
Now we apply Theorem \ref{MAIN} to the IFS constructed in Lemma \ref{lm:newIFS}.
Then we have
\[
\sum_{j>C}\|ba_j\omega_n r^j\|^2\le C
\]
for a suitable constant $C>0$ depending on the IFS and $\e$.
This means that, in particular, 
\[
\sum_{j=0}^{k}\|ba_j\omega_n^*r^{-j}\|^2\le C
\]
for any fixed $k$, provided $n$ is sufficiently large depending on $k$.
By continuity, we conclude that
\[
\sum_{j=0}^{k}\|ba_jr^{-j}\|^2\le C
\] 
for all $k$, so
\[
\sum_{j=0}^{\infty}\|ba_jr^{-j}\|^2<\infty.
\]

By Pisot's theorem we conclude that $r^{-1}$ is Pisot and $ba_j\in\Q(r)$, which, in turn, yields $a_j\in\Q(r)$,
as required.

%%%%%%%%%%%%%%%%%%%%%%%%%%%%%%%%%%%%%%%%%%%%%%%%%%%%%%%%%%%
\subsection{Proof of Corollary \ref{cr:BS}}
%%%%%%%%%%%%%%%%%%%%%%%%%%%%%%%%%%%%%%%%%%%%%%%%%%%%%%%%%%%

By a result of Bufetov and Solomyak \cite{BS14}*{Proposition 5.5}, we have
\[
\sum_{j\ge 0}\|\omega r^j\|^2\ge c \log\log \omega.
\]
for all sufficiently large $\omega$ with some constant $c>0$.
Now the corollary follows by Theorem \ref{MAIN} applied to the IFS constructed in Lemma
\ref{lm:newIFS} (applied for any $j$ such that $a_j\neq0$).

%%%%%%%%%%%%%%%%%%%%%%%%%%%%%%%%%%%%%%%%%%%%%%%%%%%%%%%%%%%
\subsection{Proof of Corollary \ref{cr:Liou}}
%%%%%%%%%%%%%%%%%%%%%%%%%%%%%%%%%%%%%%%%%%%%%%%%%%%%%%%%%%%

Without loss of generality, we assume that $l_1\neq l_2$ or $l_1=l_2=l_3$.
We apply Lemma \ref{lm:newIFS} with $j=2$ so the new IFS satisfies $a_1'-a_2'=ba_2$
for some $b\in\Q(r)$.
Inspecting the proof of the lemma, we see that the new IFS also contains two maps whose translation
components differ by $ba_3$ with the same $b$.
Theorem \ref{MAIN} implies for any sufficiently large $\omega$ that
\[
|\wh\mu(\omega)|\le\exp(-c(\DC(ba_2\omega,r^{-1},\e)+\DC(ba_3\omega,r^{-1},\e)))
\]
for any $\e>0$ with some $c=c(\e)>0$.
Now the corollary follows by Theorem \ref{AG2}.

%%%%%%%%%%%%%%%%%%%%%%%%%%%%%%%%%%%%%%%%%%%%%%%%%%%%%%%%%%%
\section{A Salem-Zygmund type result}\label{sc:uniq}
%%%%%%%%%%%%%%%%%%%%%%%%%%%%%%%%%%%%%%%%%%%%%%%%%%%%%%%%%%%

The purpose of this section is the proof of Theorem \ref{Uniqueness}.
Our argument builds on the proof of the Salem-Zygmund Theorem as it is exposed in \cite{KL87}*{Chapter III}.
We consider IFS's with scaling factors $r^{l_1},\dots,r^{l_k}$ with $r^{-1}$ being a Pisot number and $l_1,\dots,l_k\in\Z_{>0}$ with $\gcd(l_1,\ldots,l_k)=1$. We also have the translations $a_1,\dots,a_k$. After translating and scaling the self-similar system, we can assume that $a_1=0$ and $a_2,\dots,a_k$ are algebraic integers in $\mathbb{Q}(r)$.

Our proof is based on a theorem of Meyer and Rajchman, which gives a sufficient condition that makes a set a set of uniqueness. We state here the version in \cite{BS92}*{Proposition 15.4.1}. Other sources around this result can be found in \cite{KL87}.

\begin{thm}[Meyer-Rajchman]\label{MR}
Let $F\subset\mathbb{R}$ be a compact set. Let $h$ be a group homomorphism from $\mathbb{R}$ (with group action $+$) to a locally compact abelian group $G$. Suppose that $h$ has dense image in $G$. If there is a proper compact set $K\subset G$ and a sequence of real numbers $t_N,N\geq 1$ tending to $\infty$ such that
\[
h(t_N F)\subset K
\]
for all $N\geq 1$, then $F$ is a set of uniqueness.
\end{thm}

We introduce some notation.
Let $\Omega=\{1,\dots,k\}^{\Z_{>0}}$.
For each $\omega\in \Omega$, we consider the sum
\[
S(\omega)=\sum_{i\geq 1} a_{\omega_i} r^{\sum_{j=1}^{i-1} l_{\omega_j}}.
\]
It follows directly from the definition that the set
$F:=S(\Omega)$ is the self-similar set associated to the IFS.
For each $N\in\Z_{\ge 0}$, we introduce the decomposition
\[
r^{-N}S(\omega)=\sum_{i\geq 1} a_{\omega_i} r^{-N+\sum_{j=1}^{i-1} l_{\omega_j}}=S_1^{(N)}(\omega)+S_2^{(N)}(\omega),
\]
where $S_1^{(N)}$ comprises those terms of the sum for which the exponent $-N+\sum_{j=1}^{i-1} l_{\omega_j}$
of $r$ is negative, and $S_2^{(N)}$ comprises those terms for which the exponent is non-negative.

We will use the following lemma to satisfy the conditions in Theorem \ref{MR}.

\begin{lma}\label{lm:uniqueness}
There is constant $C>0$ and a compact set $L\subset\R$ of measure $0$ depending only on the IFS, and
for each $R\in\R_{\ge 1}$, there is a number $\g_R\in\Q(r)$, such that the following holds for all $N\in\Z_{\ge 0}$
and $\omega\in\Omega$
\begin{align}
\g_R\le& CR^{d-1},\label{eq:claim1}\\
\|\g_Rr^{-j} S_1^{(N)}(\omega)\|\le& C R^{-1}\text{ for any $j\ge 0$},\label{eq:claim2}\\
S_2^{(N)}(\Omega)\subset& L,\label{eq:claim3}
\end{align}
where $d$ is the degree of $r$.
\end{lma}

Before proving the lemma, we show how to finish the proof of Theorem \ref{Uniqueness}.
In Theorem \ref{MR}, we take the group homomorphism $h:\mathbb{R}\to\mathbb{T}^d$ defined by
\[
h(x)=(x,r^{-1}x,\dots,r^{-(d-1)}x)\mod \mathbb{Z}^d.
\]
As $1,r^{-1},\dots, r^{-(d-1)}$ must be linearly independent over the field of rational numbers,
we conclude that $h(\mathbb{R})$ is dense in $\mathbb{T}^d$.
Furthermore, we take $t_N:=\g_Rr^{-N}$ for suitably large fixed $R$ to be chosen below.
It follows from Lemma \ref{lm:uniqueness}, that $h(t_N F)$ is contained in the closed $CR^{-1}$-neighborhood
of the set $h(\g_R L)$.
We take this neighborhood to be $K$ in Theorem \ref{MR}.
It is, therefore, enough to show that $K\neq\T^d$ provided we choose $R$ sufficiently large.
We note that the Lebesgue measure of $K$ is at most
\[
O(R^{-d} N(\g_R L, R^{-1}))\le O( R^{-d} N(L,R^{-d})),
\]
where $N(X,s)$ denotes the minimal number of intervals of length $s$ needed to cover the set $X$.
Now the theorem follows because $L$ is a compact set of measure $0$.

\begin{proof}[Proof of Lemma \ref{lm:uniqueness}]
Write $\s_1=\Id,\s_2,\ldots,\s_{d_1}$ for the embeddings of $\Q(r)$ to $\R$
and $\s_{d_1+1},\ldots,\s_{d_1+d_2}$ for the complex embeddings of $\Q(r)$ taking one
from each pair of complex conjugates.
Then, as it is well known, $(\s_1,\ldots, \s_{d_1+d_2})$ embeds the integers in $\Q(r)$
as a lattice in $\R^d\cong\R^{d_1}\times\C^{d_2}$.

Let $\g_R$ be an algebraic integer in $\Q(r)$ such that $|\g_R|\le CR^{d-1}$ and
$\s_j(\g_R)\le CR^{-1}$ for $j\ge2$.
If $C$ is sufficiently large then these constraints prescribe a symmetric convex body in $\R^{d_1}\times\C^{d_2}$ whose volume
exceeds $2^d$ times the co-volume of the lattice that is the image of the integers in $\Q(r)$.
Thus, the existence of $\g_R$ follows from Minkowski's theorem, see e.g. \cite{Cas}*{Section III.2.2}.
Claim \eqref{eq:claim1} follows immediately.

We note that
\[
\g_R r^{-j}S_1^{(N)}(\omega)=\g_R\sum_{n=0}^\infty b_{\omega,j,N}(n) r^{-n},
\]
where $b_{\omega,j,N}(n)$ equals to $0$ or $a_i$ for some $i$ and it is non-zero for only finitely many $n$ for any fixed
values of $\omega$, $j$ and $N$.  
We write $B$ for a number that is larger than $\s_j(a_i)$ for any $i$ and $j$.
Then
\[
|\s_k(\g_R r^{-j}S_1^{(N)}(\omega))|\le B|\s_k(\g)|\sum_{n=0}^{\infty}|\s_k(r)|^{-n}\le O(R^{-1})
\]
for all $k\ge 2$.
(Here we used that $r^{-1}$ is Pisot, so $|\s_k(r)^{-1}|<1$.)
Now claim \eqref{eq:claim2} follows from the fact that $\g_R r^{-j}S_1^{(N)}(\omega)$ is an algebraic integer and hence
\[
\sum_{k=1}^{d}\s_k(\g_R r^{-j}S_1^{(N)}(\omega))\in\Z,
\]
where $\s_{d_1+d_2+1},\ldots,\s_d$ stands for the complex conjugates of $\s_{d_1+1},\ldots,\s_{d_1+d_2}$.

We turn to the proof of claim \eqref{eq:claim3}.
Fix $N$ and let $\Theta$ be the collection of all finite sequences
\[
\theta=(\theta_1,\ldots,\theta_m)\in\{1,\ldots,k\}^m
\]
such that $l_{\theta_1}+\ldots+l_{\theta_m}\ge N$, but $l_{\theta_1}+\ldots+l_{\theta_{m-1}}<N$.
We write $\theta.\omega$ for the concatenation of $\theta\in\Theta$ and $\omega\in\Omega$.
We consider $\theta.\omega$ an element of $\Omega$.

We note that
\[
S_2^{(N)}(\Omega)=\{S_2^{(N)}(\theta.\omega):\theta\in \Theta, \omega\in\Omega\}.
\]
Now we fix $\theta=(\theta_1,\ldots,\theta_m)\in\Theta$.
Then a direct calculation yields
\[
\{S_2^{(N)}(\theta.\omega):\omega\in\Omega\}=f(F),
\]
where $F$ is the self-similar set associated to the IFS and
\[
f(x)=r^{-N+l_{\theta_1}+\ldots+l_{\theta_m}} x.
\]
Indeed,
\[
S_2^{(N)}(\theta.\omega)=\sum_{i\ge 1}a_{\omega_i} r^{-N+\sum_{j=1}^{m}l_{\theta_j}+\sum_{j=1}^{i-1}l_{\omega_j}}
=f(S(\omega))
\]
for all $\omega\in\Omega$.
We note that there are finitely many possibilities for $f$ (independently of $N$), hence we can take $L$
to be the finite union of the images of $F$ under all these maps.
Since $F$ is assumed to be of measure $0$, the claim follows.
\end{proof}

\section*{Acknowledgements}
We are very grateful to Julien Br\'emont, Tom Kempton, Jialun Li and the anonymous referee for helpful comments and suggestions.

\bibliography{bib}

\end{document}